\documentclass{ajour}
\usepackage{latexsym,amsfonts}
\usepackage{amssymb}
\usepackage{amsmath}

 \def\Alt{\mathop{\rm Alt}\nolimits}

\newcommand{\C}{\mathbb{C}}    
\newcommand{\lhand}[2]{{\footnotesize\begin{array}{r} #1 \\ #2\end{array}  }\!}

\begin{document}

\title{Hurwitz generation of the universal covering of $\Alt(n)$}

\author{M. A. Pellegrini}
\affil{Dipartimento di Matematica e Applicazioni, Universit\`a degli Studi di Milano-Bicocca, Via R. Cozzi, 53, 20125 Milano Italy}
\email{marco.pellegrini@unimib.it}
\author{M. C. Tamburini} 
\affil{Dipartimento di Matematica e Fisica, Universit\`a Cattolica del Sacro Cuore, \\via Musei, 41, 25121 Brescia Italy}
\email{c.tamburini@dmf.unicatt.it}

\dedication{Dedicated to John S. Wilson on the occasion of his 65-th birthday}

\authorrunninghead{ }
\titlerunninghead{ }

\abstract{We prove that the universal covering of an alternating group $\Alt(n)$ which is Hurwitz is still Hurwitz, with 
$31$ exceptions, $30$ of which are detectable by the genus formula.}

\section{Introduction}

A finite group is Hurwitz if it can be generated by two elements of respective orders $2$ and $3$,
whose product has order 7. 
In \cite{C}  M. Conder  has constructed a $(2,3,7)$-generating triple of the alternating group $\Alt(n)$,  
for all $n>167$, and has indicated the exact values of $n\leq 167$ for which $\Alt(n)$ is Hurwitz:
they are displayed in Table 1 below.
This result was a key step in the field, and allowed further progress, e.g., the discovery that
very many linear groups over f.g. rings are $(2,3,7)$-generated (see \cite{LT}, \cite{LTW} and
 \cite{V}, for example).
In this paper, applying Conder's method, we prove the following:

\begin{theorem}\label{main}
The universal covering $\widetilde{\Alt(n)}$ of an alternating group $\Alt(n)$ which is Hurwitz, 
is still Hurwitz, except for the following values of $n$:
$$\begin{array}{ccccccccccc}
15 & 21 & 22 & 29 & 37 & 45 &  52& 71 & 79& 86& 87 \\
 94 &  101 & 102& 109& 116 & 117 & 124 & 132 & 143 & 151 & 158 \\
159 & 166 & 173&174&181 &188&215&223 &230
\end{array}$$
\end{theorem}

Apart from the case $n=21$, these exceptions are due to the failure of inequality (\ref{one}).
This inequality follows from Scott's formula \cite{S} or, equivalently, from the genus formula
\cite[Corollary page 82]{C1}: 
\begin{equation}\label{Conder}
n=84(g-1)+21r+28s+36t
\end{equation}
where $g\geq 0$, and $r,s,t$ are the numbers of fixed points of a Hurwitz generating 
triple of a transitive group of degree $n$. Hence these formulas essentially discriminate
the alternating group from its universal covering.

For each degree $n$ with positive answer, we exhibit a $(2,3,7)$-generating triple
of $\widetilde{\Alt(n)}$, up to an element of its center. 
Comparison of our generators with those of Conder provides further evidence that the same alternating group 
may well admit non-conjugate Hurwitz generators. 

Our proofs are computer independent, but the algebraic  softwares Magma and GAP have been of invaluable help.  \\

\bigskip
{\bf Acknowledgements} We are very grateful to A. Zalesskii, who suggested this problem, predicting its answer in connection to Scott's formula. We are also indebted to M. Conder who provided  Hurwitz generators for $\Alt(96)$. 

\section{Preliminary results}
For the definition and general properties of universal coverings we refer to the book of M. Aschbacher \cite [Section 33]{A}. To our purposes it is enough to recall that, for $n \geq 8$, $\widetilde{\Alt(n)}$ is perfect,  its center $\widetilde{Z}$ has order 2 
and the factor group $\widetilde{\Alt(n)}/\widetilde{Z}$ is isomorphic to $\Alt(n)$.

\begin{theorem}\label{Asch} Let $\widetilde x$ be a $2$-element in the universal covering $\widetilde{\Alt(n)}$, whose image  
$x$ in $\Alt(n)$ is an involution. Then $\widetilde x$ has order $2$ if $x$ is the product of
$4k$ cycles, has order $4$ if $x$ is the product of
$4k+2$ cycles. 
\end{theorem}

The previous result, which is part of Proposition 33.15 in \cite{A}, reduces our problem to $\Alt(n)$. Indeed:

\begin{corollary}\label{mainobserv} $\widetilde{\Alt(n)}$ is Hurwitz if and only if
$\Alt(n)$ admits a $(2,3,7)$-generating triple 
$(x,y,xy)$ in which $x$ is  the product of $4k$ cycles.
\end{corollary}
\begin{proof}
Let $x,y$ be as in the statement. 
Any preimage $\tilde x$ of $x$ in $\widetilde{\Alt(n)}$ has order 2 by Theorem \ref{Asch}.  
Clearly $y$  has a preimage  $\tilde y$ of order 3.  If  $(\tilde x\tilde y)^7=-I$, the central 
involution of $\widetilde{\Alt(n)}$, we substitute $\tilde x$ with $-\tilde x$ 
so that $(-\tilde x \tilde y)^7=I$. As both groups $\left\langle \tilde x,\tilde y\right\rangle$  and  
$\left\langle -\tilde x,\tilde y\right\rangle$ map onto $\Alt(n)$, each of them coincides with the 
group $\widetilde{\Alt(n)}$, as it is perfect . We conclude that this
group admits a $(2,3,7)$--generating triple.
The converse is obvious.
\end{proof}
\medskip
Table 1 shows the  values of 
$n<168$ for which $\Alt(n)$ is Hurwitz:
this classification appears in \cite{C}
(mention of $139$, which does not satisfy (\ref{correction}), is omitted there).

\begin{center}
\begin{tabular}{|cccccccccccccc|}\hline
\multicolumn{14}{|c|}{\bf Table 1}\\ \hline
 &15&&&&&&21&22&&&&&\\
 28&29&&&&&&35&36&37&&&&\\
 42&43&&45&&&&49&50&51&52&&&\\
 56 &57 &58&&&&&63 & 64 & 65 & 66&&&\\
 70 &71 & 72&73&&&&77 & 78 & 79&80&81&&\\
 84 & 85 & 86 & 87&88&&&91 &92 & 93 & 94&&96&\\
 98 & 99 &100 & 101 & 102&&&105 & 106 & 107 & 108 & 109&&\\
 112 & 113& 114 & 115  &116  &117&&
 119  &120 & 121 & 122 & 123 & 124&\\
 126 & 127 & 128 & 129 & 130  &&132&
 133 & 134 & 135 & 136 & 137 & 138& \\
 140 &141 & 142 & 143 & 144 & 145&&
 147 & 148 & 149 & 150 & 151 & 152 &153\\
 154 & 155  &156 & 157 & 158 & 159 & 160&
 161 & 162  &163 & 164  &165 & 166&\\
 \hline
\end{tabular}
\end{center}

\section{Negative Results}
Assume that $\Alt(n)$ is Hurwitz. It follows that
\begin{equation}\label{correction}
2\Big[\frac{n}{4}\Big]+2\Big[\frac{n}{3}\Big]+6\Big[\frac{n}{7}\Big]\geq 2n-2.
\end{equation}
Similarly, if $\widetilde {\Alt(n)}$ is Hurwitz, then 
\begin{equation}\label{one}
4\Big[\frac{n}{8}\Big]+2\Big[\frac{n}{3}\Big]+6\Big[\frac{n}{7}\Big]\geq 2n-2.
\end{equation}
These inequalities follow almost immediately from \eqref{Conder},
but also from Scott's formula (details can be found in \cite[page 399]{TV}). 
To this respect, it is useful to note that
an involution $x\in \Alt(n)$ is the product 
of $\ell \leq 2[\frac{n}{4}]$ disjoint 2-cycles and,
if $x$ is the image an involution of $\widetilde {\Alt(n)}$, then 
$\ell \leq 4[\frac{n}{8}]$. 

The values of $n$ for which $\Alt(n)$ is Hurwitz, but do not satisfy \eqref{one} are:
$$
\begin{array}{ccccccccccccccc}
15 & 22 & 29 & 37& 45& 52& 71 & 79 & 86& 87& 94 &101&102&109 & 116\\ 
117 & 124& 132&143&151&158 & 159& 166 & 173&174&181 &188&215&223 &230  \\
\end{array}
$$

\begin{lemma}
The covering $\widetilde{\Alt(21)}$ is not Hurwitz.
\end{lemma}

\begin{proof}
By contradiction let $(x,y,xy)$ be the image in $\Alt(n)$ of a $(2,3,7)$-generating triple 
of $\widetilde{\Alt(21)}$. It follows that $x$ fixes at least 5 points (actually $5$ by the genus formula).
Let $\C^{21}$ be the natural permutational module for  $\Alt(n)$ 
and $V$ its irreducible 20-dimensional component. 
Consider the diagonal action of $H=\left\langle x,y\right\rangle$ on the symmetric square $S$ of $V$
and, for $h\in H$, denote by $d^h_S$ the dimension of the space of points fixed by $h$.
Then Scott's formula gives:
$$d^x_S+d^y_S+d^{xy}_S\leq \frac{20\cdot 21}{2}+2.$$ 
Again, for details see \cite[Lemma 2.2 ]{VZ} or \cite[page 400]{TV}.
On the other hand we have $d^x_S\geq 114$,  $d^y_S\geq 70$ and $d^z_S\geq 30$,
whence the contradiction $114+70+30=214\leq 212$.
\end{proof}

\section{Proof of the result for almost all degrees}

Let $T(2,3,7)=\left\langle X,Y\mid X^2=Y^3=(XY)^7=1\right\rangle $ be the infinite triangle group.
In Conder's paper a permutation representation
$\mu: T(2,3,7)\to \Alt(m)$ is depicted by a diagram $M$, say, with $m$ vertices.
It will be convenient to say that $x=\mu(X)$ and $y=\mu(Y)$ are defined by $M$.
Assume that two vertices $j\neq k$ of $M$ form an  $(i)$-handle,  for some $i\leq 6$. This means that 
$j$ and $k$ are fixed by $x$ and that $(xy)^i$ takes $j$ to $k$. 
The following property is used repeatedly. 
Let $\mu': T(2,3,7)\to \Alt(m')$ be another representation, depicted by
$M'$. If $M'$ has  an $(i)$-handle $j',k'$, one obtains 
a new representation $T(2,3,7)\to \Alt(m+m')$ by extending the action:
$$X\mapsto \mu(x)\mu'(x)(j,j')(k,k'),\quad Y\mapsto  \mu(y)\mu'(y).$$
The diagram which depicts this representation is denoted by $M(i)M'$.

So the starting point of \cite{C} is a list of basic diagrams.
The corresponding transitive permutation representations that will be used here are given explicitly  
in \cite[Appendix A]{TV} and \cite[Appendix A]{V}.

\begin{lemma}\label{modification} 
In the notation of \emph{\cite{TV}}, let $x,y$ be defined by diagram $G$,
with  vertices $\left\{1,\dots ,42\right\}$ and $(1)$-handles
$\left\{2,3\right\},\ \left\{14,15\right\},\ \left\{32,33\right\}$.
Set 
$$x'=x\thinspace (14,32)(15,33) .$$
Then the product $x'y$ and the commutator 
$(x',y)$ are respectively conjugate to $xy$ and $(x,y)$.
\end{lemma}

\begin{proof}
Direct calculation shows that $(15,33)$ conjugates $xy$ to $x'y$, and
$(35,17,31,32,34,16,37,28,30,21,20,8,18,25,10,27,23,24,41)$ conjugates $(x,y)$ to 
$(x',y)$.  We note that the cycle structure of both commutators
$(x,y)$ and $(x',y)$ are $(2,\dots ,1,\dots  )(14,\dots ,13,\dots  )(32,\dots
,31,\dots  ) 1^3$, where each non-trivial cycle has length 13.
\end{proof}

As $x'y$ has order $7$, we may denote by $G'$ the diagram which depicts the representation 
$X\mapsto x'$, $Y\mapsto y$ of the previous Lemma.

\medskip
In Table 2 we list each basic diagram that will be needed, with its degree,
and the number $m$ of $2$-cycles of the corresponding involution $x$. 
When  a suitable power of the commutator $(x,y)$ is a cycle of prime length $p$,
we indicate explicitly this prime (also called \lq\lq useful\rq\rq). We use the notation of 
\cite[Appendix A]{TV}  for the diagrams 
called $G$, $A$, $E$, $H_i$ ($0\leq i\leq 13$),  the notation of \cite{C} 
for $B$, $C$, $D$, $J$  and that of \cite[Appendix A]{V} for the remaining ones.

\begin{center}
\begin{tabular}{|cccc|cccc|}\hline
\multicolumn{8}{|c|}{\bf Table 2}\\\hline
Diagram   & $\deg$ & $m$ & $p$ & Diagram   & $\deg$ & $m$ & $p$\\\hline
$A$ & $14$ & $6$ & {} & $C$ & $21$ & $8$ & {} \\
$B$ & $15$ &  $6$ & {}& $E$  &  $28$ & $12$ & {}\\
$D$ & $22$ & $10$ & {} &  $G'$ & $42$ & $20$ & {}\\
$G$ & $42$ & $18$ & {} & $H_2$ & $ 142$ & $68$ & $23$ \\
$H_0$ & $42$ & $18$ & $17$ & $H_3$ & $115$ & $56$ & $17$\\   
$H_1$ & $57$ & $26$ & $5$  & $H_5$ & $187$ & $92$ & $43$ \\
$H_4$ & $144$ & $70$ & $17$ & $H_7$ & $77$ & $36$ & $17$\\
$H_6$ & $216$ & $106$  & $5$ & $H_8$ & $36$ & $16$ & $5$\\
$H_{10}$& $136$ & $66$ & $5$ & $H_9$ & $135$ & $64$ & $19$\\
$J$ & $72$ & $34$ & {} & $H_{11}$ & $165$ & $80$ & $19$ \\
$O$ & $7$ & $2$ & {} & $H_{12}$ & $180$ & $88$ & $47$ \\
$P$ & $15$ & $6$& {} & $H_{13}$ & $195$ & $96$ & $23$ \\
$R$ & $22$ & $10$  & {} & $Q$ & $21$ & $8$ & {} \\
{} & {} & {} & {} & $S$ & $36$ &$ 16$ & {} \\
{} & {} & {} & {} & $T$ & $66$ & $32$& {}\\\hline
\end{tabular}
\end{center}

For each $H_i$ in Table 2, we define three composite diagrams, namely:
$$\Omega_{0}^i=H_i(1)G,\quad \Omega_{1}^i=\lhand{H_i(1)}{A(1)}G,\quad \Omega_{2}^i=\lhand{H_i(1)}{E(1)}G.$$
Here we mean that $\Omega_{1}^i$ is obtained by two joins. The first is done via the $(1)$-handle $\{2,3\}$ of $G$ 
and the $(1)$-handle of $H_i$. The second 
via the $(1)$-handle $\{14,15\}$ of $G$ and  the $(1)$-handle of $A$. Similarly for $\Omega_{2}^i$.

We are now ready to prove Theorem \ref{main} for all values of $n$ of shapes:
\begin{equation}\label{shape}
\begin{array}{ccc}
n= 42+d ,\hfill&&\\
\noalign{\medskip}
n= 42r+14s+d,& r \geq 2,& s=0,1,2,
\end{array}
\end{equation}
where $d$ is the degree of the unique diagram $H_i$ such that $n\equiv i\pmod {14}$.
For any such $n$, there exists a  composite diagram of degree $n$: e.g. one of the diagrams
$\Omega_0^i$ or 
$$\Omega_j^i(1)\underbrace{G(1)\dots (1)G}_{r-1\ {\rm times}},\quad j:=0,1,2.$$
If $x,y$ are defined by this diagram, then $\left\langle x,y\right\rangle$ is a primitive subgroup of $\Alt(n)$ 
and a power of the commutator $(x,y)$ is a $p$-cycle of prime length $p\leq n-3$ (see \cite{C}).
By a result of Jordan \cite{J}, $\left\langle x,y\right\rangle =\Alt(n)$.

Let $m$ be the number of  $2$-cycles of $x$. If $m\equiv 0\pmod 4$, by  Corollary  \ref{mainobserv}
the group $\widetilde{\Alt(n)}$ is Hurwitz. Otherwise we may 
consider the diagram obtained substituting the last copy of $G$ by $G'$.
The number of  $2$-cycles of the involution $x'$ defined by this modified diagram is 
$m+2\equiv 0\pmod 4$. 
It follows from Lemma \ref{modification} that the cycle structure of  $(x',y)$ is the same of
$(x,y)$. This allows to conclude that $\left\langle x',y\right\rangle = \Alt(n)$ by the same argument
used for $\left\langle x,y\right\rangle$.
 
Note that every $n\geq 300$ and the values listed below have shape \eqref{shape}.
$$
\begin{array}{cccccccccccccc}
78 & 84 &  99 & 119 & 120& 126& 134& 140& 141& 148& 154& 155 & 157& 161\\
162 & 168 & 169& 175 & 176& 177&  178 &  182&183& 184& 186 & 189& 190& 196\\
197& 199 & 203& 204& 207 & 210& 211& 213 & 217& 218& 219& 220 &  222 & 224\\
225& 226& 227 & 228& 229& 231& 232& 233& 234& 237 &238& 239& 240& 241 \\
242& 245& 246& 247 & 248& 249& 252& 253& 254 & 255& 256& 258 & 259& 260 \\
261& 262& 263& 264& 266 & 267& 268&269& 270 & 271& 273& 274& 275& 276\\
 277& 278& 279& 280& 281& 282 &283& 284 & 285&287& 288& 289& 290& 291\\
292& 293& 294& 295& 296& 297 & 298 &299\\
\end{array}
$$

\section{The remaining cases}\label{pos}

In this Section, for each remaining degree $n$, we give a diagram which defines a $(2,3,7)$ triple
$(x,y,xy)$ such that $x$ is the product of $m\equiv 0\pmod 4$ disjoint $2$-cycles, 
$\left\langle x,y\right\rangle$ is a primitive subgroup of 
$\Alt(n)$ which contains a $p$-cycle (called useful) of prime length $p\leq n-3$. 
As above we conclude that the group $\widetilde{\Alt(n)}$ is 
Hurwitz.

We first consider diagrams which 
involve an $H_i$, for some $i$ with $0\leq i\leq 13$. In accordance with Table 2,
it is convenient to split this interval in the two subsets: 
$$I_1=\{0,1,4,6,10\},\quad I_2=\{2,3,5,7,8,9,11,12,13\}.$$
So, for the number $m$ of $2$-cycles defined by an $H_i$,
we have $m\equiv 2\pmod 4$ if $i\in I_1$ and $m\equiv 0\pmod 4$ if 
$i\in I_2$. 

Hence, for each $i_1\in I_1$ and each $i_2\in I_2$, we consider the diagrams:
$$H_{i_1}(1)E,\quad H_{i_2}, \quad  O(1)H_{i_2}, \quad A(1)H_{i_2},\quad R(1)H_{i_2} $$
of respective degrees $\deg \left(H_{i_1}\right)+28$, $\deg \left(H_{i_2}\right)+k$ $(k=0,7,14,22)$.

They provide the following 
values of $n$ (omitting those already obtained):
$$\begin{array}{cccccccccccccccc}
36 &  43 & 50 & 58 & 70 & 77 &  85 & 91 &  115 & 122 & 129 & 135 & 137 & 142 \\
149 & 156 & 164 & 165 & 172 & 179 & 180 & 187& 194& 195 &201& 202& 209&  244.
\end{array}$$

Similarly, for $i_1\in I_1$ and $i_2\in I_2$,  the diagrams
$$\lhand{H_{i_1}(1)}{E(1)}G,\quad \lhand{H_{i_2}(1)}{A(1)}G,\quad P(1)G(1)H_{i_2},\quad \lhand{A(1)}{A(1)}G(1)H_{i_2}$$
of respective degrees $\deg \left(H_{i_1}\right)+42+28$, $\deg \left(H_{i_2}\right)+42+k$ $(k=14,15,28)$
give the new values:
$$\begin{array}{cccccccccccccc}
92 & 93& 106 &  112&  127& 133 & 147 & 171&  185 &   191& 192 & 198& 205\\
 206&  212& 214& 221& 235 & 236&  243& 250 & 251& 257 &  265 &286 .
\end{array}$$
Moreover the diagrams
$$P(1)H_3,\;\; P(1)H_9,\;\;\lhand{R(1)}{H_8(1)}G, \;\; \lhand{A(1)}{P(1)}G(1)H_8, \;\; \lhand{A(1)}{R(1)}G(1)H_8,\;\; \lhand{P(1)}{P(1)}G(1)H_8,$$
give:
$$\begin{array}{cccccc}
100 & 107& 108& 114 & 130 & 150.
\end{array}$$

For each of these diagrams, a suitable power of the commutator $(x,y)$ is the $p$-cycle
listed in Table 2, associated to the $H_i$ involved by the diagram.

Next we consider the diagrams listed in Table 3, where
the $p$-cycle is the word described in the fourth column.

\begin{center}
\begin{tabular}{|cccc|} \hline 
\multicolumn{4}{|c|}{\bf Table 3}\\ \hline
$n$ & Diagram & $p$ & $p$-cycle \\ \hline
$66$ & $T$ &  $47$ & $(xy^2xyxy^2xyxy)^{44}$\\
$28$ & $O(1)Q$ & $13$ & $(xy^2xyxyxy^2)^{24}$ \\
$35$ & $O(1)E$ & $17$ & $(xy^2xyxy^2xy^2xy^2xyxy)^{77}$\\
$42$ & $A(1)E$ & $11$ & $(xy^2xyxy^2xyxy^2xyxy)^{60}$\\
$49$ & $O(1)G'$ & $19$ & $(x,y)^{13}$ \\
$51$ & $P(1)H_8$ & $11$ & $(x,y)^{100}$ \\
$57$ & $P(1)G'$ & $23$ & $(xy^2xy^2xy^2xyxyxy^2xy^2xyxy)^{70}$\\
$63$ & $G(1)C$ & $13$ & $(xy^2xyxy^2xyxy^2xyxy)^{210}$ \\
$64$ & $R(1)G'$ & $17$ & $(xyxy^2xyxyxy^2xyxy^2)^{30}$\\
$73$ & $O(1)T$ & $47$ & $(xyxy^2xy^2xy^2xy^2xyxyxy^2xyxy^2xyxy^2xy^2)^{84}$ \\
$80$ & $A(1)T$ & $23$ & $(xy^2xyxyxy^2xyxy^2xyxy^2xyxy^2xy)^{168}$ \\
$81$ & $P(1)T$ & $67$ & $(xyxy^2xy^2xyxy^2xyxy^2xy^2xyxyxy^2xy)^7 $ \\ 
$88$ & $R(1)T$ & $71$ & $(xyxy^2xy^2xyxyxy^2xyxy^2xy)^{12}$ \\
$123$ & $H_1(1)T$ & $23$ & $(xyxy^2xy^2xyxy^2xy^2xy)^{1872}$\\
$138$ & $J(1)T$ & $13$ & $(xyxy^2xy^2xy)^{228}$ \\ 
$105$ & $C(1)G(1)G$ & $19$ & $(xyxy^2xy^2xy^2xyxy^2xy^2xyxy)^{210}$\\
$121$ & $O(1)J(1)G'$ & $17$  &   $(x,y)^{17160}$  \\
$128$ & $A(1)J(1)G'$ & $17$  & $(xyxy^2xy^2xyxyxy^2xyxy^2xy)^{390}$\\
$136$ & $R(1)J(1)G'$ & $23$ & $(xyxy^2xyxy^2xyxy^2xy)^{11970}$\\
$145$ & $O(1)J(1)T$ &  $17$  &   $(x,y)^{8360}$\\
$152$ & $A(1)J(1)T$ & $83$ & $(xyxy^2xy^2xyxyxy^2xy^2xyxy^2xyxy^2)^{828} $\\
$153$ & $P(1)J(1)T$ & $53$ & $(xyxy^2xy^2xyxy^2xyxy^2)^{690}$\\
$160$ & $R(1)J(1)T$ & $11$ & $(xyxy^2xy^2xyxyxy^2xyxy^2xy)^{9300}$\\
$163$ & $A(1)J(1)H_7$ & $17$ & $(x,y)^{3960}$\\ 
$98$ & $\lhand{A(1)}{A(1)}G(1)E$ & $19$ & $(xyxy^2xy^2xyxy^2xy^2xyxy)^{660}$\\
$113$ & $\lhand{A(1)}{P(1)}G(1)G'$ & $23$&  $(xyxy^2xy^2xy^2xyxy^2xyxy^2xyxyxy)^{70}$ \\
$144$ & $\lhand{A(1)}{R(1)}G(1)T$ & $61$ & $(xyxy^2xyxy^2xyxy^2xy)^{690}$ \\
$170$ & $\lhand{A(1)}{J(1)}G(1)G'$ & $23$ & $(xyxy^2xyxy^2xyxy)^{5460}$ \\
$193$ & $\lhand{A(1)}{R(1)} G(1)H_3$ & $29$ & $(xyxy^2xyxy^2xyxy)^{2520}$\\
$200$ & $\lhand{R(1)}{R(1)}G(1)J(1)G' $& $47$ & $(xyxy^2xy^2xyxy^2xyxy^2)^{6930}$ \\
$208$ & $\lhand{A(1)}{A(1)}G(1)J(1)T$ & $7$ &  $(xyxy^2xyxy^2xyxyxy^2)^{150}$ \\
$216$ & $\lhand{A(1)}{R(1)}G(1)J(1)T$ & $7$ & $(xyxy^2xyxy^2xyxy^2xy)^{330}$\\
$272$ & $\lhand{R(1)}{R(1)}G(1)J(1)J(1)G'$ & $17$ & $(xyxy^2xyxyxy^2xyxy^2xy^2xy^2xy)^{155610}$\\\hline
\end{tabular}
\end{center}
Using the $(2)$-handles of the diagrams $B$, $D$ and $S$ we solve the cases $n=65,72$. 
Indeed, we construct the diagrams $B(2)S(1)A$ and $D(2)S(1)A$: it turns out that in the first case the word 
$$(xyxy^2xy^2xyxyxy^2xy^2xyxy^2xy)^{3}$$
is a $59$-cycle and in the second case the word 
$$(xyxy^2xy^2xyxy^2xy^2xyxy^2xyxy^2xyxy)^{140}$$
is a $41$-cycle.
 We conclude with the last two cases $n=56, 96$, giving explicitly the generators.

\medskip
\underline{$n=56:$} 
\begin{eqnarray*} x & =&  (1,52)(2,6)(3,7)(4,53)(5,9)(8,12)(10,15)(11,13)(14,18)(16,21)\\
{} & {} & (17,22)(19,24)(20,34)(23,27)(25,30)(26,32)(28,33)(29,41)(31,36)\\
{} & {}&(35,54)(37,42)(38,40)(39,45)(43,48)(44,49)(46,51)(47,56)(50,55).\\
y & = & \prod_{i=0}^{16}  (3i+1,3i+2,3i+3) .
\end{eqnarray*}
Useful prime: $p=41$; $p$-cycle: $(xyxyxy^2xy^2xyxy^2xyxy^2xy^2)^{13}$.

\medskip
\underline{$n=96$} ({\cite{C2}):
\begin{eqnarray*}
x & = & (1, 2)(3, 4)(5, 7)(6, 10)(8, 13)(9, 16)(11, 19)(12, 14)(15, 22)(17, 25)\\
{} & {} & (18, 28)(20, 23)(21, 31)(24, 30)(26, 34)(27, 37)(29, 35)(32, 33)(36, 40)\\
{} & {} &(38, 43)(39, 46)(41, 48)(42, 49)(44, 52)(45, 55)(47, 58)(50, 56)(51, 53)\\
{} & {} & (54, 61)(57, 64)(59, 67)(60, 70)(62, 63)(65, 72)(66, 68)(69, 73)(71, 76)\\
{} & {} & (74, 79)(75, 82)(77, 85)(78, 88)(80, 90)(81, 91)(83, 89)(84, 86)(87, 94)\\
{} & {} & (92, 93)(95, 96). \\
y & = & \prod_{i=0}^{31} (3i+1,3i+2,3i+3) . 
\end{eqnarray*}
Useful prime: $p=59$; $p$-cycle: $(xyxy^2xyxyxy^2xyxy^2)^{420}$.


\begin{thebibliography}{15}

\bibitem{A} M. Aschbacher, \emph{Finite group theory} (Cambridge University Press 1986).

\bibitem{C} M. D. E. Conder. Generators for alternating and symmetric groups. \emph{J. London Math. Soc.} (2) {\bf 22} (1980), 75--86.

\bibitem{C1} M. D. E. Conder. Some results on quotients of triangle groups,  \emph{Bull. Austral. Math. Soc.} (29) {\bf 22} (1984), 73--90.

\bibitem{C2} M. D. E. Conder. Private communication.

\bibitem{LT} A. Lucchini and M. C. Tamburini,
{\it Classical groups of large rank as Hurwitz groups}, J. Algebra
{\bf 219} (1999), 531--546.

\bibitem{LTW} A. Lucchini, M. C. Tamburini and J. S. Wilson,
{\it Hurwitz groups of large rank}, J. London Math. Soc. (2) {\bf
61} (2000), 81-92.

\bibitem{J} C. Jordan. Sur la limite de transitivit\'e des groupes non altern\'es. \emph{Bull. Soc. Math. France} {\bf 1} (1872/73), 40--71. 

\bibitem{S} L. L. Scott. Matrices and cohomology. \emph{Ann. Math.} {\bf 105} (1977), 473-492.

\bibitem{TV} M. C. Tamburini and  M. Vsemirnov. Hurwitz groups and Hurwitz generation, Invited Chapter.
\emph{Handbook of Algebra}, vol. 4, edited by M. Hazewinkel (Elsevier 2006), 385--426.

\bibitem{VZ} R. Vincent and A. E. Zalesski. Non-Hurwitz classical groups, \emph{LMS J. Comput. Math.} {\bf 10} (2007), 21 -- 82.

\bibitem{V} M. Vsemirnov. Hurwitz groups of intermediate rank. \emph{LMS J. Comput. Math.} {\bf 7} (2004), 300--336 .

\end{thebibliography}
\end{document}